\documentclass[reqno]{amsart}
\usepackage{tabu}
\usepackage{amssymb}
\usepackage{mathtools}
\usepackage{a4wide,amsmath}
\usepackage{mathrsfs}
\usepackage{amsthm}
\numberwithin{equation}{section}
\numberwithin{figure}{section}
\numberwithin{table}{section}
\usepackage{bbm}
\usepackage{subfig}
\usepackage{enumerate}
\usepackage[section]{placeins}
\usepackage{graphicx}		  % Tilføjelse af billedfiler
\usepackage{ifpdf}
\ifpdf
	\DeclareGraphicsExtensions{.pdf,.eps,.jpg,.png}	
	\usepackage[suffix=]{epstopdf}
\fi
\usepackage{xcolor}
\usepackage[utf8]{inputenc}
\usepackage{hyperref}
\hypersetup{hidelinks}

\long\def\MSC#1\EndMSC{\def\arg{#1}\ifx\arg\empty\relax\else
     {\narrower\noindent%
{2010 Mathematics Subject Classification}: #1\\} \fi}
\long\def\PACS#1\EndPACS{\def\arg{#1}\ifx\arg\empty\relax\else
     {\narrower\noindent%
{PACS numbers}: #1}\fi}
\long\def\KEY#1\EndKEY{\def\arg{#1}\ifx\arg\empty\relax\else
	{\narrower\noindent% 
Keywords: #1\\}\fi}

\theoremstyle{plain}
\newtheorem{theorem}{Theorem}[section]
\newtheorem{lemma}[theorem]{Lemma}

\theoremstyle{definition}
\newtheorem{definition}[theorem]{Definition}

\newtheorem{assumption}[theorem]{Assumption}
\theoremstyle{remark}
\newtheorem{remark}[theorem]{Remark}
\newcommand{\abs}[1]{\lvert#1\rvert} 
\newcommand{\inner}[1]{\langle#1\rangle}

\newcommand{\dist}{\mathop{\textup{dist}}}
\newcommand{\essinf}{\mathop{\textup{ess\,inf}}}
\newcommand{\esssup}{\mathop{\textup{ess\,sup}}}

\newcommand{\rum}[1]{\mathbb{#1}}

\newcommand{\di}{\mathrm{d}}   % differential
\newcommand{\betaL}{\beta_\textup{L}}
\newcommand{\betaU}{\beta_\textup{U}}

\begin{document}

\title[Reconstruction of layered conductivities in EIT]{Reconstruction of piecewise constant layered conductivities in electrical impedance tomography}

\author[H.~Garde]{Henrik Garde}
\address[H.~Garde]{Department of Mathematical Sciences, Aalborg University, Skjernvej 4A, 9220 Aalborg, Denmark.}
\email{henrik@math.aau.dk}

\begin{abstract}
	This work presents a new \emph{constructive} uniqueness proof for Calder\'on's inverse problem of electrical impedance tomography, subject to local Cauchy data, for a large class of piecewise constant conductivities that we call \emph{piecewise constant layered conductivities} (PCLC). The resulting reconstruction method only relies on the physically intuitive monotonicity principles of the local Neumann-to-Dirichlet map, and therefore the method lends itself well to efficient numerical implementation and generalization to electrode models \cite{Garde_2019,GardeStaboulis_2016}. Several direct reconstruction methods exist for the related problem of inclusion detection, however they share the property that ``holes in inclusions" or ``inclusions-within-inclusions" cannot be determined. One such method is the monotonicity method of Harrach, Seo, and Ullrich \cite{Harrach10,Harrach13}, and in fact the method presented here is a modified variant of the monotonicity method which overcomes this problem. More precisely, the presented method abuses that a PCLC type conductivity can be decomposed into nested layers of positive and/or negative perturbations that, layer-by-layer, can be determined via the monotonicity method. The conductivity values on each layer are found via basic one-dimensional optimization problems constrained by monotonicity relations.   
\end{abstract}

\maketitle

\KEY
electrical impedance tomography, 
partial data reconstruction,
piecewise constant coefficient,
monotonicity principle.
\EndKEY

\MSC
35R30, % PDE: inverse problems
35Q60, % PDE: optics and electromagnetic theory
35R05, % PDE: discontinuous coefficient or data
47H05. % Operator theory: monotone operators and generalizations 
\EndMSC

\section{Introduction}

Let $\Omega\subset \rum{R}^d$, $d\geq 2$, be a bounded domain with piecewise $\mathscr{C}^\infty$-smooth boundary $\partial\Omega$ (without cusps), for which $\rum{R}^d\setminus\overline{\Omega}$ is connected. We denote by $\nu$ an outer unit normal on $\partial\Omega$, and $\Gamma\subseteq\partial\Omega$ is a non-empty relatively open subset whose role is to employ local Cauchy data. For an electrical conductivity coefficient 
\begin{equation*}
	\sigma \in L_+^\infty(\Omega) := \{ \varsigma \in L^\infty(\Omega;\rum{R}) \mid \essinf \varsigma > 0\}
\end{equation*}
and boundary current density 
\begin{equation*}
	f\in L^2_\diamond(\Gamma) := \{ g\in L^2(\Gamma) \mid \int_{\Gamma} g \, \di S = 0\} 
\end{equation*}
we consider the \emph{partial data conductivity problem}
\begin{equation} \label{eq:condeq}
	\nabla\cdot(\sigma\nabla u) = 0 \quad\text{in } \Omega, \qquad \nu\cdot \sigma\nabla u|_{\partial\Omega} = \begin{cases}
		f &\quad\text{on } \Gamma, \\
		0 &\quad\text{on } \partial\Omega\setminus\Gamma.
	\end{cases}
\end{equation}
From standard elliptic theory there is a unique solution $u = u_f^\sigma$ to \eqref{eq:condeq}, representing the interior electric potential, belonging to the ``$\Gamma$-mean free" Sobolev space
\begin{equation*}
	H^1_\diamond(\Omega) := \{ w\in H^1(\Omega) \mid \int_{\Gamma} w|_\Gamma \, \di S = 0\}.
\end{equation*}

This gives rise to a well-defined local Neumann-to-Dirichlet (ND) operator $\Lambda(\sigma) : f\mapsto u|_\Gamma$ which in this work is interpreted as a compact self-adjoint operator in $\mathscr{L}(L^2_\diamond(\Gamma))$, the space of bounded linear operators on $L^2_\diamond(\Gamma)$. 

The inverse problem of electrical impedance tomography (EIT), in the sense of Calder\'on's formulation \cite{Calderon1980}, is:
\begin{equation*}
	\textit{Reconstruct } \sigma \textit{ from knowledge of } \Lambda(\sigma).
\end{equation*}
In the practical setting, this corresponds to finding the conductivity coefficient in the interior of an object from indirect measurements of current--voltage pairs (injected current and measured voltage) recorded at electrodes placed on the object's surface. Hence, $\Lambda(\sigma)$ represents the ideal datum for such a problem. This paper will provide a new simple reconstruction method for recovering a large class of piecewise constant conductivities from their corresponding local ND map. However, first we review some known results on uniqueness and reconstruction in EIT. 

For full boundary data ($\Gamma = \partial\Omega$) unique recovery of $\sigma$ from $\Lambda(\sigma)$, i.e.\ injectivity of $\sigma\mapsto\Lambda(\sigma)$, has been solved in high generality. See e.g.~\cite{Astala2006a} for general $L^\infty_+(\Omega)$-conductivities in dimension two, and \cite{CaroRogers2016} for Lipschitz conductivities in dimension three and beyond. For full boundary data there are also reconstruction methods, based on the works of e.g.~\cite{Nachman1988a,Nachman1996,Brown1997}, such as the $\bar{\partial}$-method which has received much attention regarding theoretical development and practical implementation \cite{Siltanen2000,Knudsen2007,Knudsen2009,Cornean2006,Hamilton_2014,Siltanen_2014,Hyvonen_2018}. The motivation behind this paper stems from the expectation that, with enough restrictions on the considered class of conductivities, more straightforward and intuitive reconstruction methods will emerge. This expectation is supported by recent promising computational results in \cite{Beretta_2018}, based on shape optimization for piecewise constant conductivities on polygonal partitions.

For the different types of partial data problems in EIT (partial Dirichlet and/or Neumann data on various parts of the boundary) we refer to the review paper \cite{Kenig_2014} and the references therein. Here we will focus on local Cauchy data, in the sense of the local ND map defined above. The uniqueness problem is treated in \cite{Imanuvilov2010,Imanuvilov_2015} in two dimensions and for certain three-dimensional geometric shapes in \cite{Isakov2007,Kenig_2013}. Although for piecewise analytic conductivities the uniqueness result holds in all reasonable geometric shapes via \cite{Kohn1985,Harrach_2019}. Even when uniqueness holds for the partial data problem, exact reconstruction methods are scarce. In fact to the author's knowledge, the only other proven reconstruction method (besides the one given in this paper) is found in \cite{Nachman2010} which does not apply to local Cauchy data, but requires Dirichlet and Neumann data to be applied on a (slightly overlapping) partition of $\partial\Omega$.

We refer to the review papers \cite{Borcea2002a,Borcea2002,Cheney1999,Uhlmann2009} and references therein for more information on the theoretical and practical aspects of EIT, and refer to the list of references in section~\ref{sec:notationandlemmas} on the related problem of inclusion detection.

\begin{figure}[htb]
	\centering
	\includegraphics[width=\textwidth]{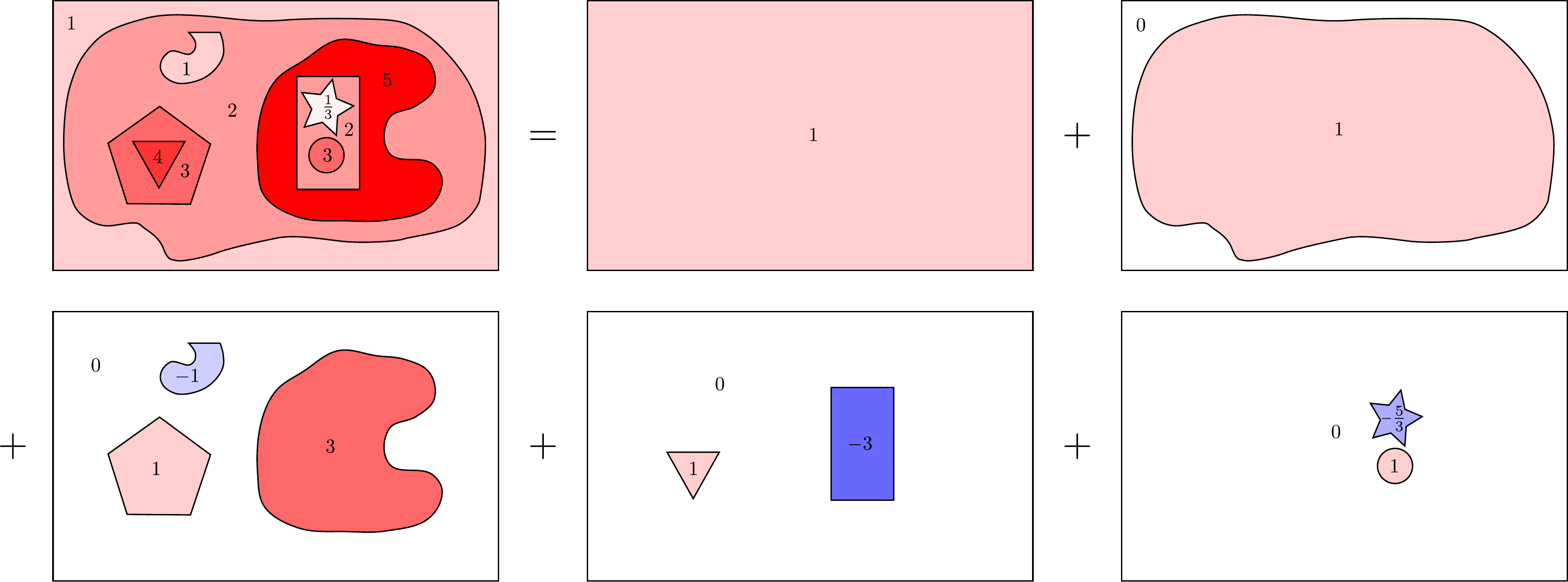}
	\caption{Decomposition of a PCLC type conductivity (top left) into each of its layers. The numbers represent function values in each of the colored regions.}
	\label{fig:fig1}
\end{figure}

In this paper, we will consider a class of piecewise constant conductivity coefficients that can be decomposed into a sum of piecewise constant functions on nested sets (layers) with connected complement. We call such a conductivity coefficient of type \emph{piecewise constant layered conductivity} (PCLC), formally defined in Definition~\ref{def:pclc} in section~\ref{sec:setting}. As illustrated by the example in Figure~\ref{fig:fig1}, this type of decomposition is in fact possible for many piecewise constant functions. The purpose of this paper is to provide a reconstruction method, based on a short and comparatively non-technical proof, that determines any PCLC type conductivity $\gamma$ from its local ND map $\Lambda(\gamma)$ via the monotonicity principles of $\sigma\mapsto\Lambda(\sigma)$. 

It is noted that \cite{Alessandrini2018} have used similar ideas with piecewise constant coefficients on layered sets. Their contribution is a uniqueness proof for the complicated case of \emph{anisotropic} piecewise constant coefficients on layered sets. The result of \cite{Alessandrini2018} is non-constructive, each layer consists of a single connected component, and they need stronger assumptions on the boundaries of the layers. Hence the result of \cite{Alessandrini2018} differ considerably from the results presented here, where the main contribution is a \emph{constructive} proof from partial boundary data. 

The remainder of the paper is organized as follows. Section~\ref{sec:setting} introduces the main assumptions and the PCLC coefficients that can be reconstructed. Section~\ref{sec:notationandlemmas} introduces some additional notation and mention two lemmas on the monotonicity principles of $\sigma\mapsto \Lambda(\sigma)$ and on localizing solutions to \eqref{eq:condeq}, that will be used for proving the main results. The main results Theorem~\ref{thm:findsupp} and Theorem~\ref{thm:findconst} are stated and proved in section~\ref{sec:findsupp} and section~\ref{sec:findconst}, respectively. Section~\ref{sec:monorecon} summarizes the actual reconstruction method based on Theorem~\ref{thm:findsupp} and Theorem~\ref{thm:findconst}. Finally, section~\ref{sec:simpleexample} is dedicated to illustrating that the method becomes quite straightforward if each layer only consists of a single connected component.

\section{The setting} \label{sec:setting}

Before giving a precise definition of PCLC type conductivities, we will start by defining the (closed) $\tau$-thinning and the outer $\tau$-layer of a set $E\subseteq \rum{R}^d$ as
\begin{align}
	H_\tau(E) &:= \{ x\in E \mid \dist(x,\partial E) \geq \tau \}, \label{eq:tauthinning}\\
	F_\tau(E) &:= \{ x\in E \mid \dist(x,\partial E) < \tau \}. \label{eq:taulayer}
\end{align}
We now state a list of assumptions on a family of sets that will be used to represent layers of a conductivity coefficient.
\begin{assumption} \label{assump}
Let $\tau > 0$, $N\in\rum{N}$, and $\{D_j\}_{j=1}^N$ be sets in $\mathbb{R}^d$ satisfying:
\begin{enumerate}[(i)]
	\item $D_j$ is the closure of a non-empty open set with piecewise $\mathscr{C}^\infty$-smooth boundary.
	\item $D_j$ has connected complement $\rum{R}^d\setminus D_j$.
	\item $D_{j+1} \subseteq H_\tau(D_j)$ for $j = 1,\dots,N-1$ and $D_1\subset \Omega$.
	\item Each set $D_j$ consists of finitely many connected components $\{D_{j,n}\}_{n=1}^{N_j}$.
\end{enumerate} 
\end{assumption}
Before continuing, we give a few remarks on these assumptions.
\begin{remark}[Related to Assumption~\ref{assump}] \label{remark:assump} {}\
	\begin{enumerate}[(1)]
		\item While we do not allow cusps on $\partial\Omega$, there can be cusps on $\partial D_j$. This is because piecewise analytic functions allow cusps on interior interfaces~\cite[Section~3]{Kohn1985}.
		\item The case $\Gamma = \partial\Omega$ allows $D_1\subseteq \overline{\Omega}$ with only minor modifications to the proof of Theorem~\ref{thm:findsupp}.
		\item Using $D_j$ as the closure of an open set, compared to a more general closed set, has the following immediate advantage: $B\cap D_j$ contains a non-empty open set for every open neighborhood $B$ of $x\in D_j$. This avoids some obvious pathological cases in the proof of Theorem~\ref{thm:findsupp}.
		\item Each connected component $D_{j,n}$ obviously also satisfies (i) and (ii) of Assumption~\ref{assump} and $\dist(\partial D_{j,n},D_{j+1})\geq\tau$.
		\item We will refer to $\tau>0$ as the \emph{minimal thickness} related to $\{D_j\}_{j=1}^N$.
		\item The layering of the sets and $\tau>0$ is required for the proofs to be \emph{constructive}. Much milder conditions apply when obtaining non-constructive uniqueness and stability proofs via monotonicity-based arguments~\cite{Harrach_2019}.
	\end{enumerate}
\end{remark} 

For a set $E\subseteq\rum{R}^d$ let $\chi_E$ denote the characteristic function on $E$. We now define the PCLC type conductivities.
\begin{definition} \label{def:pclc}
	Suppose $\{D_j\}_{j=1}^N$ satisfy Assumption~\ref{assump} with minimal thickness $\tau>0$, then we call $\gamma$ a \emph{piecewise constant layered conductivity} (PCLC), provided that
	\begin{equation*}
		\gamma = c_0 + \sum_{j=1}^N\sum_{n=1}^{N_j} c_{j,n} \chi_{D_{j,n}}
	\end{equation*}
	where $c_0>0$ and $c_{j,n} \in\rum{R}\setminus\{0\}$ satisfy $0 < \betaL \leq \gamma \leq \betaU$ in $\Omega$ for scalars $\betaL$ and $\betaU$. Here $D_j$ is called the $j$'th layer of $\gamma$, with $D_0 := \overline{\Omega}$ denoting the $0$'th layer.
\end{definition}

For $k\in \{0,1,\dots,N\}$ we define the $k$'th layer-truncated conductivity:
\begin{equation}
\gamma_k := c_0 + \sum_{j=1}^k\sum_{n=1}^{N_j} c_{j,n} \chi_{D_{j,n}}, \label{eq:gammak}
\end{equation}
where in particular $\gamma = \gamma_N$. Note that Assumption~\ref{assump} implies that $\gamma_k$ is piecewise analytic (see e.g.~\cite[Definition~2.1]{Harrach13} and \cite[Section 3]{Kohn1985}). In the following we will devise an iterative reconstruction method that at its $k$'th iteration exactly reconstructs $\gamma_k$, and naturally terminates at $k = N$. Purely from a notational point of view, in the following sections we will use $D_{N+1} := \emptyset$, which naturally is the conclusion from the $(N+1)$'th iteration. 

To summarize the ideas behind the proofs of the main results, consider the problem of determining $\gamma_{k+1}$ from $\gamma_k$ and $\Lambda(\gamma)$. This consists of two parts related to the results of Theorem~\ref{thm:findsupp} and Theorem~\ref{thm:findconst}:
\begin{enumerate}[(i)]
	\item First we find the set $D_{k+1}$. In fact, we find the components of $D_{k+1}$ inside each of the components $D_{k,n_0}$ separately.
	\item Afterwards we determine the constants $c_{k+1,m_0}$, related to each component $D_{k+1,m_0}$.
\end{enumerate} 

Part (i) focuses on reconstructing the components of $D_{k+1}$ that reside inside a component $D_{k,n_0}$. By finding certain upper bounds $D_{k+1}\cap D_{k,n_0}\subseteq C$, we may subsequently shrink $C$ until we exactly capture the set $D_{k+1}\cap D_{k,n_0}$. The monotonicity principles of $\sigma\mapsto \Lambda(\sigma)$, combined with a localization result, characterize when $C$ is an upper bound. This is done by explicitly constructing two families of operators $T_{k,n_0}^+(C)$ and $T_{k,n_0}^-(C)$, only based on $\gamma_k$ and $\Lambda(\gamma)$, such that $D_{k+1}\cap D_{k,n_0}\subseteq C$ if and only if both $T_{k,n_0}^+(C)$ and $T_{k,n_0}^-(C)$ are positive semi-definite. $D_{k+1}\cap D_{k,n_0}$ can comprise several connected components, some related to positive parts of $\gamma_{k+1}-\gamma_k$ and others related to negative parts. Therefore we need both operators $T_{k,n_0}^+(C)$ (handles positive parts) and $T_{k,n_0}^-(C)$ (handles negative parts) in order to find $D_{k+1}\cap D_{k,n_0}$. The construction of $T_{k,n_0}^+(C)$ and $T_{k,n_0}^-(C)$ is such that only the components of $D_{k+1}$ inside $D_{k,n_0}$ influence the positive semi-definiteness, i.e.\ other components of $D_{k+1}$ can be marginalized in this regard.

The ideas of part (ii) are actually very similar to those of part (i). Now we focus on a single component $D_{k+1,m_0}$, and construct two new families of operators $S_{k,m_0}^+(s)$ and $S_{k,m_0}^-(t)$, only based on $\gamma_k$, $\Lambda(\gamma)$, and $D_{k+1}$. The two families of operators are characterized by considering either a positive or negative perturbation to $\gamma_k$ on the outer $\tau$-layer of $D_{k+1,m_0}$ (hence the need for the $\tau$-thickness between the layers), while simultaneously marginalizing other components of $D_{k+1}$ and the $\tau$-thinned part of $D_{k+1,m_0}$ when it comes to positive semi-definiteness of $S_{k,m_0}^+(s)$ and $S_{k,m_0}^-(t)$. Monotonicity principles of $\sigma\mapsto\Lambda(\sigma)$ can first determine the sign of $c_{k+1,m_0}$, and afterwards find its value via a one-dimensional optimization problem constrained by positive semi-definiteness of either $S_{k,m_0}^+(s)$ (for positive sign) or $S_{k,m_0}^-(t)$ (for negative sign).

For this reconstruction method the following is assumed known/unknown a priori:
\begin{itemize}
	\item The following are assumed to be \emph{known} a priori: $\Omega$, $\Gamma$, $\Lambda(\gamma)$, $c_0$, and $\gamma$ is of type PCLC with known lower and upper bounds $\betaL$ and $\betaU$ and minimal thickness $\tau$.
	\item The following are \emph{unknown} a priori: $c_{j,n}$, $D_{j,n}$, $N_j$, and $N$. 
\end{itemize}
\begin{remark}
	Here we assume $c_0$ is known a priori. Such an assumption is also often imposed on other reconstruction methods such as the $\bar{\partial}$-method, which can be circumvented by first applying another method to reconstruct $\gamma$ on $\Gamma$, see e.g.~\cite{Nakamura2001a}. 
\end{remark}

\section{Notational remarks and lemmas} \label{sec:notationandlemmas}

For brevity we denote the essential infimum/supremum $\essinf \varsigma$ and $\esssup \varsigma$ of a function $\varsigma\in L^\infty(\Omega;\rum{R})$ by $\inf(\varsigma)$ and $\sup(\varsigma)$, respectively. $\inner{\cdot,\cdot}$ will always denote the usual $L^2(\Gamma)$-inner product.

Let $\mathscr{L}(X,Y)$ be the space of bounded linear operators between Banach spaces $X$ and $Y$, with the shorthand notation $\mathscr{L}(X) := \mathscr{L}(X,X)$. For a self-adjoint operator $T\in \mathscr{L}(L^2_\diamond(\Gamma))$ then $T\geq 0$ denotes that $T$ is a positive semi-definite operator, i.e.\ $\inner{Tf,f}\geq 0$ for all $f\in L^2_\diamond(\Gamma)$.

We will often use the symbols ``$+$"/``$-$" to associate sets and operators to positive/negative perturbations. To avoid excessive repetition, ``$\pm$" will indicate that a statement holds for both the ``$+$" and ``$-$" version of the set/operator. For example, $\smash{T_{k,n_0}^{\pm} \geq 0}$ means that both $\smash{T_{k,n_0}^+ \geq 0}$ and $\smash{T_{k,n_0}^- \geq 0}$ hold true.  

The reconstruction method will be derived based on the following two results, the \emph{monotonicity principle} and \emph{localized potentials} (which is related to the Runge approximation property), both of which are well-known results for monotonicity-based reconstruction of the support of perturbations (inclusion detection) and for non-constructive uniqueness and stability proofs in EIT, cf.~e.g.~\cite{Kang1997b,Ikehata1998a,Tamburrino2002,Gebauer2008b,Harrach10,Harrach13,Harrach15,Harrach_2019,CandianiDardeGardeHyvonen2019,GardeStaboulis_2016,Garde_2019,Garde_2017a}. 
\begin{lemma}[Monotonicity principle] \label{lemma:mono}
	For $f\in L^2_\diamond(\Gamma)$ and $\sigma_1, \sigma_2 \in L^\infty_+(\Omega)$, it holds
	\begin{equation*}
	\int_{\Omega} \frac{\sigma_2}{\sigma_1}(\sigma_1-\sigma_2) \abs{\nabla u_f^{\sigma_2}}^2\, \di x \leq \inner{(\Lambda(\sigma_2) - \Lambda(\sigma_1))f,f} \leq \int_{\Omega} (\sigma_1-\sigma_2)\abs{\nabla u_f^{\sigma_2}}^2 \, \di x.
	\end{equation*}
\end{lemma}
\begin{proof}
	This type of result goes back to \cite{Kang1997b,Ikehata1998a}. See \cite[Lemma~3.1]{Harrach13} or \cite[Lemma~2.1]{Harrach10} for a proof of this version of the result, that is readily modified to the local ND map using the variational form of \eqref{eq:condeq}. See also \cite[Section~4.3]{Harrach13} for remarks on such extensions.
\end{proof}
\begin{lemma}[Localized potentials] \label{lemma:locpot}
	Let $U\subset \overline{\Omega}$ be a relatively open connected set, which intersects~$\Gamma$, and has connected complement. Let $B\subset U$ be an open non-empty set and $\sigma\in L^\infty_+(\Omega)$ piecewise analytic, then there are sequences $(f_i)\subset L^2_\diamond(\Gamma)$ and $(u_i)\subset H^1_\diamond(\Omega)$ with $u_i = u_{f_i}^\sigma$ satisfying
	\begin{equation}
	\lim_{i\to\infty} \int_B \abs{\nabla u_i}^2\,\di x = \infty \qquad \text{and} \qquad \lim_{i\to\infty}\int_{\Omega\setminus U} \abs{\nabla u_i}^2\,\di x = 0. \label{eq:locpot}
	\end{equation}
\end{lemma}
\begin{proof}
	This result and its generalizations, ultimately based on unique continuation, is the main topic of \cite{Gebauer2008b}. Furthermore, this result is a special case of \cite[Theorem~3.6 and Section~4.3]{Harrach13}, which is also stated for locally supported Neumann conditions in \cite[Lemma~2.7]{Harrach_2019}.
\end{proof} 
It is also expected that other inclusion detection methods, such as the factorization method \cite{Bruhl2001,Bruhl2000,Kirsch2008,Hanke2015,Harrach13b,Gebauer2007} or the enclosure method \cite{Ikehata1999a,Ikehata2000c,Brander_2015}, can lead to similar reconstruction methods under stronger assumptions on the constants $c_{j,n}$ and sets $D_j$.

The map $\sigma \mapsto\Lambda(\sigma)$ is nonlinear, however it is Fr\'echet differentiable with derivative $D\Lambda(\sigma;\,\cdot\,)\in\mathscr{L}(L^\infty(\Omega;\rum{R}),\mathscr{L}(L^2_\diamond(\Gamma)))$; in fact the map is analytic~\cite[Appendix~A]{GardeHyvonenKuutela2020}. For each $\sigma\in L^\infty_+(\Omega)$, $\eta\in L^\infty(\Omega)$, and $f\in L^2_\diamond(\Gamma)$ then $D\Lambda(\sigma;\eta)$ is compact, self-adjoint, and satisfies the well-known quadratic formula (cf.~e.g.~\cite[Lemma~2.5]{Harrach_2019})
\begin{equation}
\inner{D\Lambda(\sigma;\eta)f,f} = -\int_{\Omega} \eta\abs{\nabla u_f^\sigma}^2\,\di x. \label{eq:lambdaderiv}
\end{equation}
While we could completely avoid $D\Lambda$ in this work by changing the conductivities used for the monotonicity principles, $D\Lambda$ does lead to a fast numerical method that may be of much higher practical value, without lengthening any of the proofs.

As additional notation, we define the index sets $I_j := \{1,\dots,N_j\}$ for $j\in\{1,\dots,N\}$ and $I_0 := \{1\}$ as $D_{0,1} = D_0 := \overline{\Omega}$. Moreover, 
\begin{align*}
I_j^+ &:= \{ n\in I_j \mid c_{j,n} > 0 \}, \qquad D_j^+ := \cup_{n\in I_j^+} D_{j,n}, \\
I_j^- &:= \{ n\in I_j \mid c_{j,n} < 0 \}, \qquad D_j^- := \cup_{n\in I_j^-} D_{j,n},
\end{align*}
such that $D_j = D_j^+\cup D_j^-$ decomposes the set into parts with only positive and only negative perturbations, respectively.

Since each connected component $D_{j,n_0}$ of $D_j$ can contain several connected components of $D_{j+1}$, it can swiftly become notationally demanding to have a hierarchical structure of such sets. For this reason we define a function $\mathfrak{n}_j : I_{j+1}\to I_j$, $m\mapsto n$, where $n\in I_j$ is the unique integer such that $D_{j+1,m}\subset D_{j,n}$ for given $j\in\{0,\dots,N-1\}$ and $m\in I_{j+1}$.

From this point onwards it is assumed $\gamma_k$ is known for some $k\in \{0,\dots,N-1\}$ and we will obtain results that determine $\gamma_{k+1}$. Denoting the constants
\begin{equation*}
	\alpha_{k,n} := \gamma_k|_{D_{k,n}} \quad n\in I_k, \qquad \hat{\alpha}_{k,m} := \alpha_{k,\mathfrak{n}_k(m)} \quad m\in I_{k+1},
\end{equation*}
these constants will be used to define \emph{conservative} upper bounds on the possible perturbations inside the connected components of $D_k$. Thereby we avoid having to consider the actual conductivity value on all connected components simultaneously when applying the monotonicity principles. Due to Definition~\ref{def:pclc} and Assumption~\ref{assump} it clearly holds that $\betaL\leq \alpha_{k,n} \leq \betaU$ for all $n\in I_k$. Moreover, from \eqref{eq:gammak}, Definition~\ref{def:pclc}, and Assumption~\ref{assump}(iii) we obtain the following bounds for any $n_0\in I_k$:
\begin{align}
	\gamma-\gamma_k &\leq \enskip\smashoperator{\sum_{m\in I_{k+1}}} (\betaU-\hat{\alpha}_{k,m})\chi_{D_{k+1,m}} \leq \enskip\smashoperator{\sum_{n\in I_k\setminus\{n_0\}}} (\betaU-\alpha_{k,n})\chi_{D_{k,n}} + (\betaU - \alpha_{k,n_0})\chi_{D_{k+1}\cap D_{k,n_0}}, \label{eq:gammadiffup} \\
	\gamma-\gamma_k &\geq \enskip\smashoperator{\sum_{m\in I_{k+1}}} (\betaL-\hat{\alpha}_{k,m})\chi_{D_{k+1,m}} \geq \enskip\smashoperator{\sum_{n\in I_k\setminus\{n_0\}}} (\betaL-\alpha_{k,n})\chi_{D_{k,n}} + (\betaL - \alpha_{k,n_0})\chi_{D_{k+1}\cap D_{k,n_0}}. \label{eq:gammadifflow}
\end{align}
In particular, $\betaL-\alpha_{k,n}$ represents the largest possible (signed) negative perturbation that can occur within $D_{k,n}$ when determining $\gamma_{k+1}$ from $\gamma_k$, and likewise $\betaU-\alpha_{k,n}$ is the largest possible positive perturbation.

\section{Reconstruction of \texorpdfstring{$D_{k+1}$ from $\gamma_k$ and $\Lambda(\gamma)$}{next layer}} \label{sec:findsupp}

For $n_0\in I_k$ and measurable $C\subseteq \overline{\Omega}$ we now define some operators based on $\gamma_k$ and $\Lambda(\gamma)$:
\begin{align*}
	T_{k,n_0}^+(C) &:= \Lambda(\gamma) - \Lambda(\gamma_k) - \smashoperator{\sum_{n\in I_k\setminus\{n_0\}}}(\betaU - \alpha_{k,n})D\Lambda(\gamma_k; \chi_{D_{k,n}}) - (\betaU-\alpha_{k,n_0})D\Lambda(\gamma_k; \chi_{C}), \\
	T_{k,n_0}^-(C) &:= \Lambda(\gamma_k) - \Lambda(\gamma) + \smashoperator{\sum_{\substack{\phantom{x}\\{n\in I_k\setminus\{n_0\}}}}} \frac{\alpha_{k,n}}{\betaL}(\betaL - \alpha_{k,n})D\Lambda(\gamma_k; \chi_{D_{k,n}}) + \frac{\alpha_{k,n_0}}{\betaL}(\betaL-\alpha_{k,n_0})D\Lambda(\gamma_k; \chi_{C}).
\end{align*}
In fact, we will consider sets $C$ that belong to families of \emph{admissible test inclusions} relative to some subset $E\subseteq\overline{\Omega}$:
\begin{equation*}
	\mathcal{A}(E) := \{ C\subseteq \overline{E} \mid C \text{ is closed and } \rum{R}^d\setminus C \text{ is connected} \}.
\end{equation*}
In what follows these test inclusions will be used to determine $D_{k+1}$ from $\gamma_k$. Note that Theorem~\ref{thm:findsupp} below essentially corresponds to a modified version of the usual monotonicity method for indefinite inclusions, applied separately on each connected component of $D_{k}$; cf.~\cite[Theorem~2.3]{Garde_2019} and \cite[Section~4.2]{Harrach13}. 
\begin{theorem} \label{thm:findsupp}
	Let $n_0\in I_k$, then for all $C\in \mathcal{A}(D_{k,n_0})$ it holds
	\begin{equation}
		D_{k+1} \cap D_{k,n_0} \subseteq C \qquad \text{if and only if} \qquad T_{k,n_0}^{\pm}(C) \geq 0. \label{eq:findsuppcondition}
	\end{equation}
	In particular, $D_{k+1} \cap D_{k,n_0} = \cap\{ C\in \mathcal{A}(D_{k,n_0}) \mid T_{k,n_0}^{\pm}(C) \geq 0 \}$.
\end{theorem}
\begin{proof}
	First we prove the direction ``$\Rightarrow$" in the \emph{if and only if} statement. Assume $D_{k+1} \cap D_{k,n_0} \subseteq C$, then it holds by Lemma~\ref{lemma:mono}, \eqref{eq:lambdaderiv}, and \eqref{eq:gammadiffup},
	\begin{align*}
		-\inner{T_{k,n_0}^+(C)f,f} &\leq \int_{\Omega} \Bigl[ \gamma-\gamma_k - \smashoperator{\sum_{n\in I_k\setminus\{n_0\}}}(\betaU - \alpha_{k,n})\chi_{D_{k,n}} - (\betaU-\alpha_{k,n_0})\chi_{C} \Bigr]\abs{\nabla u_f^{\gamma_k}}^2\,\di x \\
		&\leq (\alpha_{k,n_0}-\betaU)\int_{C\setminus (D_{k+1} \cap D_{k,n_0})}\abs{\nabla u_f^{\gamma_k}}^2\,\di x \leq 0
	\end{align*}
	for all $f\in L^2_\diamond(\Gamma)$, i.e.\ $T_{k,n_0}^+(C)\geq 0$. 
	
	Likewise, since $\frac{\gamma_k}{\gamma} \leq \frac{\alpha_{k,n}}{\betaL}$ in $D_{k,n}$ and $\betaL\leq \alpha_{k,n}$ then Lemma~\ref{lemma:mono}, \eqref{eq:lambdaderiv}, and \eqref{eq:gammadifflow} imply
	\begin{align*}
		\inner{T_{k,n_0}^-(C)f,f} & \geq \int_{\Omega} \Bigl[ \frac{\gamma_k}{\gamma}(\gamma-\gamma_k) - \smashoperator{\sum_{\substack{\phantom{x}\\{n\in I_k\setminus\{n_0\}}}}}\frac{\alpha_{k,n}}{\betaL}(\betaL - \alpha_{k,n})\chi_{D_{k,n}} - \frac{\alpha_{k,n_0}}{\betaL}(\betaL-\alpha_{k,n_0})\chi_{C} \Bigr]\abs{\nabla u_f^{\gamma_k}}^2\,\di x \\
		&\geq \frac{\alpha_{k,n_0}}{\betaL}(\alpha_{k,n_0}-\betaL)\int_{C\setminus (D_{k+1} \cap D_{k,n_0})}\abs{\nabla u_f^{\gamma_k}}^2\,\di x \geq 0
	\end{align*}
	for all $f\in L^2_\diamond(\Gamma)$, i.e.\ $T_{k,n_0}^-(C)\geq 0$. This concludes the first part of the proof.  
	
	The proof of the other direction ``$\Leftarrow$" of the \emph{if and only if} statement is shown as a contrapositive, i.e.\ assume $D_{k+1} \cap D_{k,n_0} \not\subseteq C$ then we will in the following contradict one of the inequalities $T^{\pm}_{k,n_0}\geq 0$.
	
	We now pick a relatively open connected set $U\subset\overline{\Omega}$, which intersects $\Gamma$, has connected complement, and satisfies: $D_{k+1,m_0}\cap U$ contains an open ball $B$ for some $m_0\in I_{k+1}$ with $\mathfrak{n}_k(m_0) = n_0$ and 
	\begin{equation*}
		U\cap \left[(D_k\setminus D_{k,n_0})\cup C\cup (D_{k+1}\setminus D_{k+1,m_0})\cup D_{k+2}\right] = \emptyset. \label{eq:Uavoid}
	\end{equation*}
	The reasoning behind the properties of $U$ is: Assumption~\ref{assump} and $C\in \mathcal{A}(D_{k,n_0})$ imply the set $\overline{\Omega}\setminus[(D_k\setminus D_{k,n_0})\cup C]$ is connected and contains $\Gamma$. Moreover, $(D_{k+1} \cap D_{k,n_0}) \setminus C$ contains a non-empty open set due to (i) and (iii) of Assumption~\ref{assump} (cf.~Remark~\ref{remark:assump}). Since $D_{k+1}$ comprise finitely many closed connected components (Assumption~\ref{assump}) implies a strictly positive distance between these connected components. Thus $U$ can be chosen to only intersect one connected component of $D_{k+1}\cap D_{k,n_0}$, and furthermore avoid $D_{k+2}$ due to Assumption~\ref{assump}(iii).
	
	This splits the rest of the proof into two possible cases, related to which one of the inequalities $T_{k,n_0}^{\pm} \geq 0$ that will be contradicted:
	\begin{equation*}
		\text{(a): } m_0\in I_{k+1}^+ \qquad \text{or} \qquad \text{(b): } m_0\in I_{k+1}^-.
	\end{equation*}
	
	\emph{Case (a).} Note that $\gamma = \alpha_{k,n_0} + c_{k+1,m_0}$ in $B$ with $c_{k+1,m_0}>0$ and $\gamma\geq \gamma_k$ in $U$ (equality holds in $U\setminus D_{k+1,m_0}$). The main idea is to construct potentials $u$ via Lemma~\ref{lemma:locpot} where simultaneously $\abs{\nabla u}^2$ is large inside $B$ and small outside $U$, in such a way that Lemma~\ref{lemma:mono} contradicts the inequality $T_{k,n_0}^+(C)\geq 0$. Since $\gamma_k$ is piecewise analytic and by the properties of $U$, it follows from Lemma~\ref{lemma:locpot} that there are sequences $(f_i)\subset L^2_\diamond(\Gamma)$ of current densities and corresponding localized potentials $(u_i)\subset H_\diamond^1(\Omega)$ that solve \eqref{eq:condeq} with conductivity $\gamma_k$, and satisfy \eqref{eq:locpot}.
	
	Denoting 
	\begin{equation*}
		\hat{\gamma} := \frac{\gamma_k}{\gamma}(\gamma-\gamma_k)- \smashoperator{\sum_{n\in I_k\setminus\{n_0\}}}(\betaU - \alpha_{k,n})\chi_{D_{k,n}} - (\betaU-\alpha_{k,n_0})\chi_{C},
	\end{equation*}
	we have by Lemma~\ref{lemma:mono}, \eqref{eq:lambdaderiv}, and \eqref{eq:locpot}
	\begin{align*}
		-\inner{T_{k,n_0}^+(C)f_i,f_i} &\geq \int_B \frac{\gamma_k}{\gamma}(\gamma-\gamma_k)\abs{\nabla u_i}^2\,\di x + \int_{U\setminus B} \frac{\gamma_k}{\gamma}(\gamma-\gamma_k)\abs{\nabla u_i}^2\,\di x + \int_{\Omega\setminus U} \hat{\gamma}\abs{\nabla u_i}^2\,\di x \\
		&\geq \frac{\alpha_{k,n_0}c_{k+1,m_0}}{\alpha_{k,n_0}+c_{k+1,m_0}} \int_B\abs{\nabla u_i}^2\, \di x + \inf(\hat{\gamma}) \int_{\Omega\setminus U}\abs{\nabla u_i}^2\,\di x \to \infty \text{ for } i\to\infty,
	\end{align*}
	from which we conclude $T_{k,n_0}^+(C) \not\geq 0$.
	
	\emph{Case (b).} In this case we have $\gamma = \alpha_{k,n_0} + c_{k+1,m_0}$ in $B$ with $c_{k+1,m_0}<0$ and $\gamma\leq \gamma_k$ in $U$. Denote 
	\begin{equation*}
		\tilde{\gamma} := \gamma-\gamma_k - \smashoperator{\sum_{\substack{\phantom{x}\\{n\in I_k\setminus\{n_0\}}}}}\frac{\alpha_{k,n}}{\betaL}(\betaL - \alpha_{k,n})\chi_{D_{k,n}} - \frac{\alpha_{k,n_0}}{\betaL}(\betaL-\alpha_{k,n_0})\chi_{C}.
	\end{equation*}
	Applying the above construction of localized potentials satisfying \eqref{eq:locpot}, we contradict the inequality $T_{k,n_0}^-\geq 0$ using Lemma~\ref{lemma:mono} and \eqref{eq:lambdaderiv}:
	\begin{align*}
	\inner{T_{k,n_0}^-(C)f_i,f_i} &\leq \int_B (\gamma-\gamma_k)\abs{\nabla u_i}^2\,\di x + \int_{U\setminus B} (\gamma-\gamma_k)\abs{\nabla u_i}^2\,\di x + \int_{\Omega\setminus U} \tilde{\gamma}\abs{\nabla u_i}^2\,\di x \\
	&\leq c_{k+1,m_0} \int_B\abs{\nabla u_i}^2\, \di x + \sup(\tilde{\gamma}) \int_{\Omega\setminus U}\abs{\nabla u_i}^2\,\di x \to -\infty \text{ for } i\to\infty,
	\end{align*}
	hence concluding $T_{k,n_0}^-(C) \not\geq 0$.
	
	The equality $D_{k+1} \cap D_{k,n_0} = \cap \mathcal{M}$ with $\mathcal{M} := \{ C\in \mathcal{A}(D_{k,n_0}) \mid T_{k,n_0}^{\pm}(C) \geq 0 \}$ is satisfied via \eqref{eq:findsuppcondition} since $D_{k+1} \cap D_{k,n_0}\subseteq C$ for each $C\in \mathcal{M}$ and that $D_{k+1} \cap D_{k,n_0}$ itself is a member of $\mathcal{M}$.
\end{proof}

\section{Reconstruction of \texorpdfstring{$\gamma_{k+1}$ from $\gamma_k$, $\Lambda(\gamma)$, and $D_{k+1}$}{next layer-truncated conductivity}} \label{sec:findconst}

Now that Theorem~\ref{thm:findsupp} gives a way of determining $D_{k+1}$ from $\gamma_k$, the next step is to determine the constant $c_{k+1,m_0}$ for each $m_0\in I_{k+1}$ in order to obtain $\gamma_{k+1}$. For this purpose we define for $m_0\in I_{k+1}$, $s\in[0,\betaU-\hat{\alpha}_{k,m_0}]$, and $t\in [\betaL-\hat{\alpha}_{k,m_0},0]$ the operators
\begin{align*}
S_{k,m_0}^+(s) &:= \Lambda(\gamma) - \Lambda(\gamma_{k,m_0,\betaU} + s\chi_{F_\tau(D_{k+1,m_0})}), \\
S_{k,m_0}^-(t) &:= \Lambda(\gamma_{k,m_0,\betaL} + t\chi_{F_\tau(D_{k+1,m_0})})-\Lambda(\gamma),
\end{align*}
for which $\gamma_{k,m_0,\beta}$ with $\beta\in\{\betaL,\betaU\}$ is defined as
\begin{equation*}
\gamma_{k,m_0,\beta} := \gamma_k + \smashoperator{\sum_{m\in I_{k+1}\setminus\{m_0\}}} (\beta-\hat{\alpha}_{k,m})\chi_{D_{k+1,m}} + (\beta-\hat{\alpha}_{k,m_0})\chi_{H_\tau(D_{k+1,m_0})}. 
\end{equation*}
Recall the definition of $H_\tau$ and $F_\tau$ in \eqref{eq:tauthinning} and \eqref{eq:taulayer}. As we shall see in Theorem~\ref{thm:findconst}, there are two equivalent ways of determining if $m_0\in I_{k+1}$ belongs to $I_{k+1}^+$ or $I_{k+1}^-$. Afterwards, we may find the constant $c_{k+1,m_0}\in [\betaL-\hat{\alpha}_{k,m_0},0)\cup(0,\betaU-\hat{\alpha}_{k,m_0}]$ via an optimization problem, by varying $s$ and $t$ on the outer $\tau$-layer of $D_{k+1,m_0}$, constrained by positive semi-definiteness of $S_{k,m_0}^{\pm}$.
\begin{theorem} \label{thm:findconst}
	Let $m_0\in I_{k+1}$, then it holds
	\begin{align}
		[0,\betaU-\hat{\alpha}_{k,m_0}]\ni s \geq c_{k+1,m_0} &\qquad \text{if and only if} \qquad S_{k,m_0}^+(s)\geq 0, \label{eq:posint} \\
		[\betaL-\hat{\alpha}_{k,m_0},0]\ni t \leq c_{k+1,m_0} &\qquad \text{if and only if} \qquad S_{k,m_0}^-(t)\geq 0. \label{eq:negint}
	\end{align}
	As direct consequences,	
	\begin{align*}
	m_0\in I_{k+1}^+ &\qquad \text{if and only if} \qquad S_{k,m_0}^-(0)\geq 0 \qquad \text{if and only if} \qquad S_{k,m_0}^+(0)\not\geq 0, \\
	m_0\in I_{k+1}^- &\qquad \text{if and only if} \qquad S_{k,m_0}^+(0)\geq 0 \qquad \text{if and only if} \qquad S_{k,m_0}^-(0)\not\geq 0,
	\end{align*}
	and $c_{k+1,m_0}$ is determined via:
	\begin{equation*}
	c_{k+1,m_0} = \begin{cases}
	\min\{ s\in (0,\betaU-\hat{\alpha}_{k,m_0}] \mid S_{k,m_0}^+(s)\geq 0\} & \text{if } m_0\in I_{k+1}^+, \\
	\max\{ t\in [\betaL-\hat{\alpha}_{k,m_0},0) \mid S_{k,m_0}^-(t)\geq 0\} & \text{if } m_0\in I_{k+1}^-.
	\end{cases}
	\end{equation*} 
\end{theorem}
\begin{proof}
Note that $\gamma = \hat{\alpha}_{k,m_0} + c_{k+1,m_0}$ in the set $F_\tau(D_{k+1,m_0})$ due to Assumption~\ref{assump}(iii). Moreover, $\gamma_k = \hat{\alpha}_{k,m_0}$ in $F_\tau(D_{k+1,m_0})$, so writing 
\begin{equation*}
	\gamma-\gamma_{k,m_0,\betaU} = (\gamma-\gamma_{k,m_0,\betaU})\chi_{\overline{\Omega}\setminus F_\tau(D_{k+1,m_0})} + (\gamma-\gamma_{k,m_0,\betaU})\chi_{F_\tau(D_{k+1,m_0})}
\end{equation*}
we may apply \eqref{eq:gammadiffup} to bound the first term from above by 0. Likewise for $\gamma-\gamma_{k,m_0,\betaL}$ we obtain a lower bound using \eqref{eq:gammadifflow}, resulting in
\begin{equation} \label{eq:newgammabnd}
	\gamma-\gamma_{k,m_0,\betaU} \leq c_{k+1,m_0}\chi_{F_\tau(D_{k+1,m_0})} \leq \gamma-\gamma_{k,m_0,\betaL}.
\end{equation}	
We begin by proving \eqref{eq:posint}, hence denote the piecewise analytic $L^\infty_+(\Omega)$-function
\begin{equation*}
	\hat{\gamma} := \gamma_{k,m_0,\betaU}+s\chi_{F_\tau(D_{k+1,m_0})},
\end{equation*}
and assume $s\geq c_{k+1,m_0}$. By virtue of Lemma~\ref{lemma:mono} and \eqref{eq:newgammabnd}
\begin{align*}
	-\inner{S_{k,m_0}^+(s)f,f} &\leq \int_{\Omega} \left[ \gamma-\gamma_{k,m_0,\betaU}-s\chi_{F_\tau(D_{k+1,m_0})} \right] \abs{\nabla u_f^{\hat{\gamma}}}^2\,\di x \\
	&\leq (c_{k+1,m_0}-s)\int_{F_\tau(D_{k+1,m_0})}\abs{\nabla u_f^{\hat{\gamma}}}^2\,\di x \leq 0
\end{align*}
for all $f\in L^2_\diamond(\Gamma)$, i.e.\ $S_{k,m_0}^+(s) \geq 0$ for $s\geq c_{k+1,m_0}$.

For the opposite implication we assume $s < c_{k+1,m_0}$. In a similar way to the proof of Theorem~\ref{thm:findsupp}, we pick a relatively open connected set $U\subset\overline{\Omega}$, which intersects $\Gamma$, has connected complement, satisfies $(D_{k+1}\setminus D_{k+1,m_0})\cap U = H_{\tau}(D_{k+1,m_0})\cap U = \emptyset$, and $F_{\tau}(D_{k+1,m_0})\cap U$ contains an open ball $B$. Once again this is possible due to Assumption~\ref{assump}. Hence $\gamma-\hat{\gamma} = c_{k+1,m_0}-s > 0$ in $B$ and $\gamma\geq\hat{\gamma}$ in $U$. 

Now let $(f_i)\subset L^2_\diamond(\Gamma)$ and $(u_i)\subset H^1_\diamond(\Omega)$ be chosen via Lemma~\ref{lemma:locpot} with respect to the sets $U$ and $B$ for the conductivity $\hat{\gamma}$. Lemma~\ref{lemma:mono} gives
\begin{align*}
	-\inner{S_{k,m_0}^+(s)f_i,f_i} &\geq \int_\Omega \frac{\hat{\gamma}}{\gamma}(\gamma-\hat{\gamma})\abs{\nabla u_i}^2\,\di x \\
	&\geq \frac{\hat{\alpha}_{k,m_0}+s}{\hat{\alpha}_{k,m_0}+c_{k+1,m_0}}(c_{k+1,m_0}-s)\int_B\abs{\nabla u_i}^2\,\di x + \inf(\tfrac{\hat{\gamma}}{\gamma}(\gamma-\hat{\gamma})) \int_{\Omega\setminus U} \abs{\nabla u_i}^2\,\di x.
\end{align*}
Since $0\leq s < c_{k+1,m_0}$ then \eqref{eq:locpot} implies $\lim_{i\to\infty}\inner{S_{k,m_0}^+(s)f_i,f_i}= -\infty$. We conclude $S_{k,m_0}^+(s)\not\geq 0$ for $s<c_{k+1,m_0}$.

Next we prove \eqref{eq:negint} in an analogous way. Denote the piecewise analytic $L^\infty_+(\Omega)$-function
\begin{equation*}
	\tilde{\gamma} := \gamma_{k,m_0,\betaL} + t\chi_{F_\tau(D_{k+1,m_0})}.
\end{equation*}
First we assume $t\leq c_{k+1,m_0}$, and since $t\in[\betaL-\hat{\alpha}_{k,m_0},0]$ it holds $\frac{\tilde{\gamma}}{\gamma}\geq \frac{\betaL}{\betaU}$ in $\Omega$. Thus from Lemma~\ref{lemma:mono} and \eqref{eq:newgammabnd} it holds
\begin{align*}
	\inner{S_{k,m_0}^-(t)f,f} &\geq \int_{\Omega} \frac{\tilde{\gamma}}{\gamma}(\gamma-\tilde{\gamma}) \abs{\nabla u_f^{\tilde{\gamma}}}^2\,\di x \geq \frac{\betaL}{\betaU} (c_{k+1,m_0}-t)\int_{F_\tau(D_{k+1,m_0})} \abs{\nabla u_f^{\tilde{\gamma}}}^2\,\di x \geq 0
\end{align*}
for all $f\in L^2_\diamond(\Gamma)$, i.e.\ $S_{k,m_0}^-(t) \geq 0$ for $t\leq c_{k+1,m_0}$.

For the opposite implication we assume $t>c_{k+1,m_0}$ and pick the sets $U$ and $B$ in exactly the same way as in the proof of \eqref{eq:posint}. In particular, $\gamma-\tilde{\gamma} = c_{k+1,m_0}-t<0$ in $B$ and $\gamma\leq \tilde{\gamma}$ in $U$. Now let $(f_i)\subset L^2_\diamond(\Gamma)$ and $(u_i)\subset H^1_\diamond(\Omega)$ be chosen according to Lemma~\ref{lemma:locpot} for the sets $U$ and $B$ and with conductivity $\tilde{\gamma}$.

Applying Lemma~\ref{lemma:mono} and \eqref{eq:locpot} yields
\begin{align*}
	\inner{S_{k,m_0}^-(t)f_i,f_i} &\leq \int_\Omega (\gamma-\tilde{\gamma})\abs{\nabla u_i}^2\,\di x \\
	&\leq (c_{k+1,m_0}-t)\int_B \abs{\nabla u_i}^2\,\di x + \sup(\gamma-\tilde{\gamma})\int_{\Omega\setminus U} \abs{\nabla u_i}^2\,\di x\to -\infty \text{ for } i\to\infty,
\end{align*}
whence $S_{k,m_0}^-(t) \not\geq 0$ for $t > c_{k+1,m_0}$.
\end{proof}

\begin{remark} \label{remark:m0}
	Based on the proofs of Theorem~\ref{thm:findsupp} and Theorem~\ref{thm:findconst}, it is straightforward to show that the conclusion of whether $m_0\in I_{k+1}$ belongs to $I_{k+1}^+$ or $I_{k+1}^-$ in Theorem~\ref{thm:findconst} is preserved when replacing $S_{k,m_0}^{\pm}(0)$ with $\tilde{S}_{k,m_0}^{\pm}$ defined below, where $\tilde{D} := H_\tau(D_{k+1,m_0})$:
	\begin{align*}
		\tilde{S}_{k,m_0}^+ &:= \Lambda(\gamma) - \Lambda(\gamma_k) - \smashoperator{\sum_{m\in I_{k+1}\setminus\{m_0\}}} (\betaU-\hat{\alpha}_{k,m})D\Lambda(\gamma_k;\chi_{D_{k+1,m}}) - (\betaU-\hat{\alpha}_{k,m_0})D\Lambda(\gamma_k;\chi_{\tilde{D}}), \\
		\tilde{S}_{k,m_0}^- &:= \Lambda(\gamma_k) - \Lambda(\gamma) + \smashoperator{\sum_{\substack{\phantom{x}\\{m\in I_{k+1}\setminus\{m_0\}}}}} \frac{\hat{\alpha}_{k,m}}{\betaL}(\betaL - \hat{\alpha}_{k,m})D\Lambda(\gamma_k; \chi_{D_{k+1,m}}) + \frac{\hat{\alpha}_{k,m_0}}{\betaL}(\betaL-\hat{\alpha}_{k,m_0})D\Lambda(\gamma_k; \chi_{\tilde{D}}).
	\end{align*}
	It is tempting to also use $D\Lambda$ to apply the variation of $s$ and $t$ on $F_\tau(D_{k+1,m_0})$ in Theorem~\ref{thm:findconst}. However, the set $U$ for the localized potentials will intersect part of the set on which $D\Lambda$ is applied (unlike in the proof of Theorem~\ref{thm:findsupp}, where this is specifically avoided), and the resulting integrals do not lead to a proof of the desired assertion. 
\end{remark}

\section{Monotonicity-based reconstruction of PCLC conductivities} \label{sec:monorecon}

We can now summarize the reconstruction method based on Theorem~\ref{thm:findsupp} and Theorem~\ref{thm:findconst} in the following way:
\begin{enumerate}[(1)]
	\item Let $\gamma_k$ for some $k\in\{0,1,\dots,N\}$ be given (initially $\gamma_0 = c_0$ with $D_0 := \overline{\Omega}$).
	\item Determine $D_{k+1}$ via: for each $n_0\in I_k$ using Theorem~\ref{thm:findsupp} we find
	\begin{equation*}
		D_{k+1} \cap D_{k,n_0} = \cap\{ C\in \mathcal{A}(D_{k,n_0}) \mid T_{k,n_0}^{\pm}(C) \geq 0 \}.
	\end{equation*}
	\item For each $m_0\in I_{k+1}$ we employ Theorem~\ref{thm:findconst}/Remark~\ref{remark:m0} to determine if $m_0\in \smash{I_{k+1}^+}$ or $m_0\in \smash{I_{k+1}^-}$ by the positive semi-definiteness (or lack thereof) of either 
	\begin{equation*}
		S_{k,m_0}^+(0), \qquad S_{k,m_0}^-(0), \qquad \tilde{S}_{k,m_0}^+, \qquad {or} \qquad \tilde{S}_{k,m_0}^-.
	\end{equation*}
	\item Theorem~\ref{thm:findconst} determines $c_{k+1,m_0}$ as:
	\begin{equation*}
	c_{k+1,m_0} = \begin{cases}
	\min\{ s\in (0,\betaU-\hat{\alpha}_{k,m_0}] \mid S_{k,m_0}^+(s)\geq 0\} & \text{if } m_0\in I_{k+1}^+, \\
	\max\{ t\in [\betaL-\hat{\alpha}_{k,m_0},0) \mid S_{k,m_0}^-(t)\geq 0\} & \text{if } m_0\in I_{k+1}^-.
	\end{cases}
	\end{equation*} 
	\item The above steps determine $\gamma_{k+1}$. Repeat the above steps iteratively, until we reach $\gamma_{N+1} = \gamma_N$ by finding $D_{N+1} = \emptyset$ in step~(2), hence concluding the reconstruction method.
\end{enumerate}

\begin{remark}
	Note that numerical implementation of step~(2) above can be handled, both in terms of regularization theory and practical implementation, via a layer peeling approach \cite[Theorem~3.1 and Algorithm~1]{Garde_2019}. For other considerations in this direction see also \cite{GardeStaboulis_2016,Harrach15,Garde_2017a}. Step~(4) can be handled straightforwardly via bisection due to \eqref{eq:posint} and \eqref{eq:negint} in Theorem~\ref{thm:findconst}.
\end{remark}

\section{Simplifications when each layer only has a single connected component} \label{sec:simpleexample}

This section will illustrate the considerable simplifications to the reconstruction method, in the special case when each layer $D_j$ only consists of a single connected component. Hence the complicated expressions dedicated to marginalizing other components are no longer required.

In this situation, we may name the constants $c_j$ rather than $c_{j,n}$ and write
\begin{equation}
	\gamma_k := \sum_{j=0}^k c_j\chi_{D_j},\enskip k \in \{0,1,\dots,N\},
\end{equation}
using the convention that $D_0 := \overline{\Omega}$. Again we have $\gamma = \gamma_N$. Recall that $c_0 > 0$ is assumed known, and for each $k\in\{0,1,\dots,N-1\}$ we must reconstruct the set $D_{k+1}$ and constant $c_{k+1}\in \mathbb{R}\setminus\{0\}$ based on knowledge of $\gamma_k$ and $\Lambda(\gamma)$.

Define $\alpha_k := \sum_{j=0}^k c_j$ for each $k\in\{0,1,\dots,N\}$, and for measurable $C\subseteq \overline{\Omega}$ we define the operators
\begin{align*}
T_k^+(C) &:= \Lambda(\gamma) - \Lambda(\gamma_k) - (\betaU-\alpha_{k})D\Lambda(\gamma_k; \chi_{C}), \\
T_k^-(C) &:= \Lambda(\gamma_k) - \Lambda(\gamma) + \frac{\alpha_{k}}{\betaL}(\betaL-\alpha_{k})D\Lambda(\gamma_k; \chi_{C}).
\end{align*}
For $s\in[0,\betaU-\alpha_{k}]$ and $t\in [\betaL-\alpha_{k},0]$ we define the operators
\begin{align*}
S_{k}^+(s) &:= \Lambda(\gamma) - \Lambda\left(\gamma_{k} + (\betaU-\alpha_k)\chi_{H_\tau(D_{k+1})} + s\chi_{F_\tau(D_{k+1})}\right), \\
S_{k}^-(t) &:= \Lambda\left(\gamma_{k} + (\betaL-\alpha_k)\chi_{H_\tau(D_{k+1})} + t\chi_{F_\tau(D_{k+1})}\right) - \Lambda(\gamma).
\end{align*}
Hence, in this situation, the reconstruction method is as follows:
\begin{enumerate}[(1)]
	\item Let $\gamma_k$ for some $k\in\{0,1,\dots,N\}$ be given (initially $\gamma_0 = c_0$ with $D_0 := \overline{\Omega}$).
	\item Determine $D_{k+1}$ via:
	\begin{equation*}
	D_{k+1} = \cap\{ C\in \mathcal{A}(D_{k}) \mid T_{k}^{\pm}(C) \geq 0 \}.
	\end{equation*}
	\item The sign of $c_{k+1}$ is determined via either of:
	\begin{equation*}
	c_{k+1} < 0 \Leftrightarrow S_{k}^+(0)\geq 0, \qquad c_{k+1}>0 \Leftrightarrow S_{k}^-(0) \geq 0.
	\end{equation*}
	\item Find $c_{k+1}$ via:
	\begin{equation*}
	c_{k+1} = \begin{cases}
	\min\{ s\in (0,\betaU-\alpha_{k}] \mid S_{k}^+(s)\geq 0\} & \text{if } c_{k+1} > 0, \\
	\max\{ t\in [\betaL-\alpha_{k},0) \mid S_{k}^-(t)\geq 0\} & \text{if } c_{k+1} < 0.
	\end{cases}
	\end{equation*} 
	\item The above steps determine $\gamma_{k+1}$. Repeat the steps iteratively, until we reach $\gamma_{N+1} = \gamma_N$ by finding $D_{N+1} = \emptyset$ in step~(2), hence concluding the reconstruction method.
\end{enumerate}
\begin{remark}
	Note that in step~(3) we may also use operators of the form in Remark~\ref{remark:m0} that involve $D\Lambda$. If there is the further simplification that all the constants $c_j$ have the same sign, then only one of the operators $T_k^{\pm}$ is needed for step~(2) and only one of the operators $S_k^\pm$ is needed for step~(4).
\end{remark}

\subsection*{Acknowledgments}

This work was supported by the Academy of Finland (decision 312124) and the Aalto Science Institute (AScI). HG thanks Nuutti~Hyv\"onen for encouraging discussions during a research visit at Aalto University.

\bibliographystyle{plain}
\bibliography{minbib}

\end{document}